\documentclass[a4paper, 11pt, reqno]{amsart}

\usepackage[usenames,dvipsnames]{color}
\usepackage{amssymb, amsmath, latexsym, graphics, graphicx, pstricks}

\newcommand{\ol}[1]{{\overline{#1}}}

\newtheorem{theorem}{Theorem}[section]
\newtheorem*{maintheorem}{Main Theorem}
\newtheorem{lemma}[theorem]{Lemma}

\newtheorem{corollary}[theorem]{Corollary}

\begin{document}
\title[Anisimov's Theorem for inverse semigroups]{Anisimov's Theorem for inverse semigroups}


\maketitle

\begin{center}

 MARK KAMBITES\footnote{Email \texttt{Mark.Kambites@manchester.ac.uk}.}

    \medskip

    School of Mathematics, \ University of Manchester, \\
    Manchester M13 9PL, \ England.

 \date{\today}
\thanks{}

\end{center}

\begin{abstract}
The idempotent problem of a finitely generated inverse semigroup is
the formal language of all words over the generators representing idempotent
elements. This note proves that a finitely generated inverse semigroup
with regular idempotent problem is necessarily finite. This answers
a question of Gilbert and Noonan Heale, and establishes a generalisation
to inverse semigroups of Anisimov's Theorem for groups.
\end{abstract}

\section{Introduction}

The \textit{word problem} of a finitely generated group is the set of
all words over some given generating set which represent the identity
element in the group. Anisimov's Theorem \cite{Anisimov72} is the statement that a
finitely generated group is finite if and only if it has word problem
which is a regular language. Although in itself not difficult
(indeed, modulo some foundational results in formal language theory, the proof is
almost trivial), the theorem is of considerable significance. It laid the foundations
for the use of formal language theory in combinatorial group theory,
presaging a host of later (and much harder) correspondence theorems
describing the relationship between
structural properties of a finitely generated group, and language-theoretic
properties of its word problem. Probably the most famous and 
celebrated of these is the \textit{Muller-Schupp Theorem} \cite{Muller83}, which says that
a finitely generated group has context-free word problem if and only it
has a free subgroup of finite index; other examples can be found in
\cite{Brough11,K_counter,Holt08,Holt05,Holt08b,KambitesWordProblemsRecognisable}.

The reason the word problem is capable of capturing important structural features of a group, is that
it encapsulates in language-theoretic form the computational problem of deciding whether two words
in the generators represent the same element (that is, the word problem in the sense of universal
algebra). The word problem in this latter sense has also been studied for finitely
generated semigroups (both in full generality, and restricted to ``group-like'' classes
such as \textit{cancellative} and \textit{inverse} semigroups); indeed, it has been a major
focus of research for over 60 years, with many beautiful results and a number of important questions
remaining open (see just for example \cite{Adyan87,Birget98,Post47}). However,
the lack of identity and inverses in a semigroup means that this problem is not encapsulated in
the language of words representing the identity.
Gilbert and Noonan Heale \cite{Gilbert11} recently introduced a natural language-theoretic
generalisation of the group word
problem to finitely generated inverse monoids, by studying the language of those
words over a given finite generating set which represent \textit{idempotent}
elements of the monoid; they term this language the \textit{idempotent
problem}. (Here, we shall generalise the setting very slightly further to
consider also inverse semigroups without an identity element, but the difference
is not of great importance.)

Given this definition, the question immediately arises of whether there is an
analogue of Anisimov's Theorem, that is, does an inverse semigroup have
regular
idempotent problem if and only if it is finite? One implication ---
that a finite inverse semigroup has regular idempotent problem --- is,
as in the group case, almost immediate \cite[Proposition~3]{Gilbert11}.
However, the converse --- that every
inverse semigroup with regular idempotent problem is finite --- is not
so obvious. Gilbert and Noonan Heale established that this implication holds
for certain special classes of inverse semigroups (those which are \textit{$E$-unitary}
\cite[Theorem~6]{Gilbert11}, \textit{strongly $E$-reflexive} \cite[Theorem~7]{Gilbert11}
or have finitely many idempotents \cite[Theorem~10]{Gilbert11}), but posed as an open problem
the question of whether the implication holds in general \cite[Question 4]{Gilbert11}. The purpose of this note is
to answer this question in the affirmative, thus establishing a complete
analogue of Anisimov's Theorem for inverse semigroups:

\begin{maintheorem}[Anisimov's Theorem for Inverse Semigroups]
A finitely generated inverse semigroup or monoid has regular idempotent problem
if and only if it is finite.
\end{maintheorem}

Conceptually, the proof proceeds by reducing the problem to a purely structural
question about inverse semigroups: namely, whether a finitely generated, infinite inverse
semigroup can have a finite, idempotent-pure quotient. This question in
turn is answered using an embedding theorem of Billhardt \cite{Billhardt92}
which allows us to embed an inverse semigroup into a \textit{$\lambda$-semidirect
product} of an idempotent-pure quotient acting on a semilattice, together with a
new (although not difficult) result showing that the property of local
finiteness is preserved under taking $\lambda$-semidirect products with
finite inverse monoids. This conceptual order is reversed in the structure of
the paper which, in addition to this introduction, comprises two sections.
Section~\ref{sec_structure} establishes the required structural results concerning
inverse semigroups and idempotent-pure congruences, while Section~\ref{sec_main}
builds upon these to establish our main results.

\section{Inverse Semigroups, Congruences and $\lambda$-semidirect Products}\label{sec_structure}

We begin by recalling some elementary facts about inverse semigroups and congruences; for a fuller
introduction see for example \cite{Lawson98}.

A semigroup $S$ is called \textit{inverse} if for every element $x$ there
is a unique element $x^{-1}$ such that $x x^{-1} x = x$ and $x^{-1} x x^{-1} = x^{-1}$. The map
$x \mapsto x^{-1}$ is an involution on any inverse semigroup $S$. Inverse
semigroups have a rich and beautiful structure theory, and arise naturally in numerous
areas of both pure and applied mathematics, most notably as models of \textit{partial
symmetry} \cite{Lawson98}. Recall that an element $e$ of a semigroup is called \textit{idempotent}
if $e^2 = e$. The idempotents in an inverse semigroup $S$ are exactly the elements of the form $xx^{-1}$
for $x \in S$. Typically inverse semigroups have many idempotents; inverse semigroups with one idempotent
are exactly groups. An inverse semigroup in which \textit{every} element is
idempotent is called a \textit{semilattice} (such structures being exactly meet semilattices,
in the sense of order theory, under the meet operation).

A \textit{congruence} $\equiv$ on a semigroup $S$ is an equivalence relation $\equiv$ such that $a \equiv b$ implies
$xa \equiv xb$ and $ax \equiv bx$ for all $x \in S$; we write $[s]$ for the equivalence class of an element
$s \in S$. If $\equiv$ is a congruence then the set $S / \hspace{-3pt} \equiv$ of equivalence classes has an
associative multiplication well-defined by $[s][t] = [st]$; the map $S \to S / \hspace{-3pt} \equiv, s \mapsto [s]$ is a morphism. Conversely,
if $\rho : S \to T$ is a morphism then the relation $\cong$ given by $a \cong b \iff \rho(a) = \rho(b)$ is a
congruence on $S$, called the \textit{kernel} of $\rho$, and the image $\rho(S)$ is isomorphic to
$S / \hspace{-3pt} \cong$. A congruence is called \textit{idempotent-pure} if idempotent elements are not related
to non-idempotent elements, that is, if $e \equiv a$ and $e^2 = e$ together imply $a^2 = a$. A
morphism of semigroups is called \textit{idempotent-pure} if its kernel is idempotent-pure.

Now let $G$ be an inverse semigroup acting by endomorphisms on the left of
another inverse semigroup $A$. For clarity, we denote the multiplication in
$G$ by juxtaposition, the multiplication in $A$ by $\circ$ and the action
by $\cdot$. Consider the set
$$\lbrace (\alpha, g) \in A \times G \mid g g^{-1} \cdot \alpha = \alpha \rbrace$$
equipped with a multiplication given by 
$$(\alpha, g) (\beta, h) = \left( \left[ (gh)(gh)^{-1} \cdot \alpha \right] \circ \left[ g \cdot \beta \right], \ gh \right).$$
It can be shown \cite[Proposition~1]{Billhardt92} that the given set is closed
under this operation, and indeed forms an inverse semigroup. This semigroup, which was
introduced and first studied by Billhardt \cite{Billhardt92}, is called the
\textit{$\lambda$-semidirect product} of $G$ acting on $A$.

We remark that the $\lambda$-semidirect product is somewhat unusual, indeed arguably
even unnatural, in the context of inverse semigroup 
theory. Inverse semigroups usually arise as models of ``partial 
symmetry'', and by the Wagner-Preston Theorem \cite[Theorem 1.5.1]{Lawson98}
can all be represented as such. It is thus customary (and almost always most 
natural) to considering them acting by partial bijections, rather than by 
functions, and an action by endomorphisms thus seems intuitively like a
``category error''. However, the following theorem of
Billhardt \cite[Corollary~7]{Billhardt92} and the numerous uses to which it has been
put show that it is nevertheless extremely useful.

\begin{theorem}[Billhardt 1992]\label{thm_billhardt}
Let $I$ be an inverse semigroup and $\equiv$ be an idempotent-pure congruence
on $I$. Then $I$ embeds in a $\lambda$-semidirect product of the quotient
$I / \hspace{-3pt} \equiv$ acting on a semilattice.
\end{theorem}

We shall need the following straightforward but useful result about
$\lambda$-semidirect products. Recall that a semigroup is called \textit{locally
finite} if all of its finitely generated subsemigroups are finite.

\begin{theorem}\label{thm_localfinite}
Any $\lambda$-semidirect product of a finite inverse semigroup $G$ acting on a locally finite inverse
semigroup $A$ is locally finite. Moreover, if $\sigma : \mathbb{N} \to \mathbb{N}$ is non-decreasing such that
every
$k$-generated subsemigroup of $A$ has cardinality bounded above by $\sigma(k)$, then every $m$-generated
subsemigroup of the $\lambda$-semidirect product has cardinality bounded above by $|G| \sigma(m |G|)$.
\end{theorem}
\begin{proof}
Let $I$ be the $\lambda$-semidirect product, and let $Y$ be a
finite subset of $I$. Our first aim is to show that the subsemigroup
generated by $Y$ is finite.

Define
$$Y' = \lbrace \alpha \mid (\alpha, g) \in Y \textrm{ for some } g \in G \rbrace$$
and
$$U = \lbrace h \cdot \alpha \mid \alpha \in Y', h \in G \rbrace.$$
That is, $U$ is the union of the orbits under the $G$-action of the
$A$-components of pairs in $Y$. Let $W$ be the subsemigroup of $A$
generated by $U$. Since $Y$ is finite, $Y'$ is finite. Since $G$ is
finite, it follows that $U$ is finite. And since $A$ is locally finite,
we conclude that $W$ is finite.

We claim first that $W$ is invariant under the $G$-action on $A$. Indeed, suppose
$\gamma \in W$ and $g \in G$. Then $\gamma = \alpha_1 \circ \dots \circ \alpha_n$ for some
$n \in \mathbb{N}$ and
$\alpha_1, \dots, \alpha_n \in U$. By the definition of $U$, each $\alpha_i$ can be
written as $g_i \cdot \beta_i$ where $g_i \in G$ and $\beta_i \in Y'$. Now since $G$
acts by endomorphisms we have
$$g \cdot \gamma \ = \ g \cdot (\alpha_1 \circ \dots \circ \alpha_n) \ = \ g \cdot (g_1 \cdot \beta_1 \circ \dots \circ g_n \cdot \beta_n) \ = \ (g \cdot g_1 \cdot \beta_1) \circ \dots \circ (g \cdot g_n \cdot \beta_n).$$
But by the definition of $U$, each $g \cdot g_i \cdot \beta_i = (g g_i) \cdot \beta_i \in U$, so by
the definition of $W$, we conclude that $g \cdot \gamma \in W$, as required to prove
the claim.

Now we shall show that the subsemigroup $\langle Y \rangle$ of $I$
generated by $Y$ is contained in $W \times G$; since $W$ and $G$ are
both finite, this will show that $\langle Y \rangle$ is finite.
That $Y$ itself is contained in $W \times G$ follows immediately from
the definitions, so it suffices to show that the
product of two elements in $\langle Y \rangle \cap (W \times G)$ must lie
in $W \times G$.

Suppose, then, that $(\alpha, g)$ and $(\beta, h)$ are in 
$\langle Y \rangle \cap (W \times G)$. Then by the definition of the
$\lambda$-semidirect product, their product in $I$ is
$$\left( \left[ (gh)(gh^{-1}) \cdot \alpha \right] \circ \left[ g \cdot \beta \right],  \ gh \right).$$
Clearly the right-hand component is in $G$, and because $\alpha, \beta \in W$ and
$W$ is closed under both the multiplication in $A$ and the $G$-action, the left-hand
component lies in $W$.

Finally, suppose $\sigma : \mathbb{N} \to \mathbb{N}$ is non-decreasing such
that every $k$-generated subsemigroup of $A$ has cardinality bounded above
by $\sigma(k)$. Then by
the definitions we have $|U| \leq |Y'| |G| \leq |Y| |G|$. But $W$ is a subsemigroup of
$A$ generated by $|U|$ elements, so by the properties of $\sigma$ we have
$|W| \leq \sigma(|U|) \leq \sigma(|Y||G|)$. Now since $\langle Y \rangle$ embeds in $W \times G$, we have
$$|\langle Y \rangle| \ \leq \ |W \times G| \ = \ |G| |W| \ \leq \ |G| \sigma(|Y||G|)$$
as required to complete the proof.
\end{proof}

\begin{corollary}\label{cor_finite}
A $k$-generated inverse semigroup with a finite idempotent-pure quotient of cardinality $n$ has
cardinality at most $n (2^{kn} -1)$.
\end{corollary}
\begin{proof}
Let $I$ be a $k$-generated inverse semigroup and $J$ a
finite idempotent-pure quotient of cardinality $n$. Then by Theorem~\ref{thm_billhardt}
(Billhardt's Theorem), $I$ embeds in the $\lambda$-semidirect product
of $J$ acting on a semilattice $A$. The size of any $k$-generated semilattice, and hence
of any $k$-generated subsemigroup of $A$, 
is bounded by the size of the free semilattice of rank $k$, which is easily seen to be
$2^k - 1$. Thus, defining $\sigma : \mathbb{N} \to \mathbb{N}$ by
$\sigma(k) = 2^k - 1$, Theorem~\ref{thm_localfinite} tells us that every $k$-generated
subsemigroup of the $\lambda$-semidirect product, and in particular $I$ itself, has
cardinality at most $n (2^{kn} -1)$.
\end{proof}

We remark that the bound given by Corollary~\ref{cor_finite} is trivially
attained by a free semilattice of any rank $k$, since it has an idempotent-pure
morphism to the trivial group of
cardinality $n=1$. However, the bound seems unlikely to be sharp for $n > 1$, and it would be
interesting to see how far it can be improved.

\section{Languages and the Idempotent Problem}\label{sec_main}

In this section we shall recall the definition of the idempotent problem, and then establish our main
result. We begin by recapping some basic definitions from language theory; for a more detailed
introduction see for example \cite{Howie91}.

Let $\Sigma$ be a set (for now, not necessarily finite), called an \textit{alphabet}. We denote by $\Sigma^*$
the set of all
\textit{words} over $\Sigma$, that is, finite sequences of elements of $\Sigma$, including an \textit{empty word}
denoted $\epsilon$.
The set $\Sigma^*$ forms a monoid, with identity $\epsilon$, under the operation of concatenation (writing one word next to another),
called the \textit{free monoid} on $\Sigma$. The subset $\Sigma^* \setminus \lbrace \epsilon \rbrace$
forms a subsemigroup, called the \textit{free semigroup} on $\Sigma$, which we denote $\Sigma^+$.
A set of words over $\Sigma$, that is, a subset of $\Sigma^*$, is called a \textit{language} over $\Sigma$.  A language $L \subseteq \Sigma^*$
induces a congruence on $\Sigma$, given by $a \equiv_L b$ if and only if for all $x, y \in \Sigma^*$
we have $xay \in L \iff xby \in L$. This is called the \textit{syntactic congruence} of $L$; the
quotient $\Sigma^* / \hspace{-3pt} \equiv_L$ is the \textit{syntactic monoid} of $L$, denoted
$M(L)$, and the quotient map from $\Sigma^*$ to $M(L)$ is the \textit{syntactic
morphism}. We define the \textit{syntactic semigroup} of $L$, denoted $M^+(L)$, to be the subsemigroup
of $M(L)$ consisting of those equivalence classes containing a non-empty word, that is, the
image of $\Sigma^+$ under the syntactic morphism. Clearly either $M^+(L) = M(L)$ or $M^+(L)$ is
$M(L)$ with its identity element removed, depending on whether the empty word $\epsilon$ is
syntactically equivalent to any other word.

We shall call a language over a finite alphabet
\textit{regular} if its syntactic monoid (or, equivalently, its syntactic semigroup) is
finite. Regular languages are more usually defined
as those recognised by finite automata or generated by regular grammars; the definition we have given
is equivalent to these by a foundational result of language theory \cite[Theorem~3.1.4]{Howie91} and,
since it accords with the property which we shall use in our proof, saves defining irrelevant
notions. A language is called a \textit{group language} if its syntactic monoid is a group.

Now let $I$ be an inverse semigroup [respectively, monoid] generated by a (not necessarily finite)
subset $\Sigma$. The embedding of the generating set $\Sigma$ into $I$ extends to a canonical
surjective morphism from the free semigroup $\Sigma^+$ [respectively, the free monoid $\Sigma^*$]
onto $I$. The \textit{idempotent problem} for $I$ (with respect to the generating
set $\Sigma$), is the set of all words in $\Sigma^+$ [respectively, $\Sigma^*$] which map onto
idempotents of $I$. Note
that if $I$ is a group, the idempotent problem of $A$ is just the set of words representing the
identity in $I$, that is, the \textit{word problem} of $I$. The idempotent problem was introduced and
studied (in the case $\Sigma$ is finite and $I$ a monoid) by Gilbert and Noonan Heale \cite{Gilbert11}.

The following result gives a basic but useful structural property of the idempotent problem.

\begin{lemma}\label{lemma_syntacticmorphism}
Let $I$ be an inverse semigroup [respectively, monoid] with a (not necessarily finite) semigroup
[monoid] generating set $\Sigma$, and let $L$ be the idempotent problem of $I$ with respect to
$\Sigma$. Then there is a well-defined, idempotent-pure, surjective morphism from $I$
onto
$M^+(L)$, [respectively, $M(L)$] which takes each element of $I$ to the syntactic
equivalence class of any word over $\Sigma$ representing it. The kernel of
this map is the greatest idempotent-pure congruence on $I$.
\end{lemma}
\begin{proof}
Except where otherwise stated, we assume during this proof that words are drawn from $\Sigma^+$ in
the semigroup case and $\Sigma^*$ in the monoid case.

To show the claimed map is well-defined, we need to show that any
two words representing the same element of $I$ are syntactically
equivalent with respect to $L$. Indeed, suppose words $u$ and $v$
represent the same element of
$I$. Then clearly
for any words $x, y \in \Sigma^*$, we have that $xuy$ and $xvy$
represent the same element of $I$. In particular, $xuy$ represents an
idempotent if and only if $xvy$ represents an idempotent. That is,
$xuy \in L$ if and only if $xvy \in L$. Thus, $u \equiv_L v$.

The fact that the map is a morphism follows immediately from the fact that the canonical
morphism from $\Sigma^+$ [$\Sigma^*$] onto $I$ and the syntactic morphism from $\Sigma^+$
to $M^+(L)$ [respectively, $\Sigma^*$ to $M(L)$] are
both morphisms. The fact that the map is surjective is an immediate consequence of the fact
that the syntactic morphism maps $\Sigma^+$ [$\Sigma^*$] surjectively onto $M^+(L)$ [$M(L)$],
and every word represents some element of $I$.

To show the map is idempotent-pure, suppose $e$ and $a$ map to the same
place and $e$ is idempotent. Choose words $u$ and $v$ representing $e$
and $a$ respectively; then $u$ and $v$ are syntactically equivalent.
Since $e$ is idempotent, $u \in L$, so by syntactic equivalence $v \in L$,
and so $a$ is idempotent.

Finally, to show that the kernel is the greatest idempotent-pure congruence,
suppose for a contradiction that elements $a, b \in I$ are related by
some idempotent-pure congruence $\sigma$ but not by our kernel. Let $u$ and $v$ be words representing $a$ and
$b$ respectively. Since $a$ and $b$ are taken to different places by our
map, $u$ and $v$ are not syntactically equivalent, so there exist words
$x, y \in \Sigma^*$  such that one of $xuy$ and $xvy$ lies in $L$ and the other
does not. Let $\ol{x}$ and $\ol{y}$ be the elements of $I$ represented by $x$ and $y$
respectively This means that one of $\ol{x} a \ol{y}$ and $\ol{x} b \ol{y}$
is idempotent and the other is not. But $\sigma$ is a congruence so
$\ol{x} a \ol{y} \mathrel{\sigma} \ol{x} b \ol{y}$, which contradicts the
assumption that $\sigma$ is idempotent-pure.
\end{proof}

Lemma~\ref{lemma_syntacticmorphism} has two obvious, but nevertheless interesting, consequences.
The first is that
the isomorphism type of the syntactic semigroup or monoiod of the idempotent problem (which
encapsulates many of its language-theoretic properties) is
an invariant of the inverse semigroup or monoid which can be defined without reference
to a choice of generators (and even without the existence of a finite
choice of generators).
\begin{corollary}
The syntactic semigroup [monoid] of the idempotent problem of a (not necessarily finitely generated) inverse
semigroup [monoid] is invariant under change of semigroup [monoid] generating set.
\end{corollary}

The second interesting consequence is that any property of inverse semigroups or monoids which is visible
in the quotient by the greatest idempotent-pure congruence is also visible in the idempotent problem. For example, 
recall that an inverse semigroup is called \textit{$E$-unitary} if the product of an idempotent and
a non-idempotent element is never idempotent. This is equivalent to the minimum group congruence 
being idempotent-pure \cite[Theorem 2.4.6]{Lawson98}, which in turn is easily seen to be equivalent to the
quotient by the greatest idempotent-pure morphism being a group. Combining with
Lemma~\ref{lemma_syntacticmorphism}, it follows
that that $E$-unitary property of inverse semigroups manifests itself as a natural property of the
idempotent problem:

\begin{corollary}
An inverse semigroup or monoid is $E$-unitary if and only if its idempotent problem is a group language.
\end{corollary}



We now have all the ingredients for our main theorem.

\begin{maintheorem}[Anisimov's Theorem for Inverse Semigroups]
A finitely generated inverse semigroup or monoid has regular idempotent problem
if and only if it is finite.
\end{maintheorem}
\begin{proof}
Let $I$ be a finite generated inverse semigroup [monoid] and $L$ its idempotent problem with respect
to some finite set of generators. Suppose $L$ is regular. By Lemma~\ref{lemma_syntacticmorphism},
$I$ admits an idempotent-pure morphism to $M^+(L)$ [$M(L)$], which is finite because $L$ is regular.
Hence, by Corollary~\ref{cor_finite}, $I$ is finite. The converse, as mentioned
above, was established by
Gilbert and Noonan Heale \cite[Proposition~3]{Gilbert11}.
\end{proof}

The numerical bound from Corollary~\ref{cor_finite} says that a $k$-generated inverse semigroup
whose idempotent problem has $n$ syntactic equivalence classes cannot have more than $n (2^{kn} -1)$
elements. Again, this bound is attained for a free semilattice of rank $k$ (where $n=1$) but can
probably be improved substantially for larger $n$. If we fix a language, any inverse semigroup of
which it is the idempotent problem is $k$-generated where $k$ is number of letters appearing in the
language. It
follows that there are, up to isomorphism, only finitely many inverse semigroups with a given regular
language as their idempotent problem.

\section*{Acknowledgements}

The author thanks Tara Brough, Marianne Johnson and Markus Pfeiffer
for helpful conversations.

\bibliographystyle{plain}

\def\cprime{$'$} \def\cprime{$'$}

\end{document}